\documentclass[12pt]{amsart}
\usepackage{fullpage,url,mathrsfs,stmaryrd}
\usepackage[all]{xy}

\newtheorem{theorem}{Theorem}[section]
\newtheorem{lemma}[theorem]{Lemma}
\newtheorem{lemma-def}[theorem]{Lemma-Definition}

\newtheorem{cor}[theorem]{Corollary}

\theoremstyle{definition}

\newtheorem{question}[theorem]{Question}

\newtheorem{notation}[theorem]{Notation}
\newtheorem{remark}[theorem]{Remark}
\newtheorem{defn}[theorem]{Definition}

\numberwithin{equation}{theorem}

\newcommand{\CC}{\mathbb{C}}

\newcommand{\QQ}{\mathbb{Q}}

\newcommand{\ZZ}{\mathbb{Z}}

\newcommand{\calE}{\mathcal{E}}
\newcommand{\calF}{\mathcal{F}}

\newcommand{\calO}{\mathcal{O}}
\newcommand{\calP}{\mathcal{P}}

\newcommand{\bv}{\mathbf{v}}
\newcommand{\bw}{\mathbf{w}}
\newcommand{\bx}{\mathbf{x}}
\newcommand{\frakm}{\mathfrak{m}}
\newcommand{\frako}{\mathfrak{o}}

\newcommand{\dual}{\vee}

\DeclareMathOperator{\Ext}{Ext}

\DeclareMathOperator{\FIsoc}{\mathbf{F-Isoc}}
\DeclareMathOperator{\Gal}{Gal}

\DeclareMathOperator{\Hom}{Hom}

\DeclareMathOperator{\rank}{rank}

\DeclareMathOperator{\Spec}{Spec}

\begin{document}

\title{A cut-by-curves criterion for overconvergence of $F$-isocrystals}
\author{Thomas Grubb, Kiran S. Kedlaya, and James Upton}
\thanks{Kedlaya was supported by NSF (DMS-1802161, DMS-2053473) and UC San Diego (Warschawski Professorship).}

\maketitle

Throughout, let $k$ be a finite field of characteristic $p>0$ and let $X$ be a smooth, geometrically irreducible $k$-scheme.
For each prime $\ell \neq p$, one can consider \'etale cohomology of $X$ with $\overline{\QQ}_\ell$-coefficients, for which the category of lisse Weil sheaves 
provides a natural theory of coefficient objects of locally constant rank.
For $\ell = p$, one gets a satisfactory analogue only by replacing \'etale cohomology with rigid cohomology, which has \emph{two} different but related analogues of the category of lisse Weil sheaves: the category of \emph{convergent $F$-isocrystals} and the subcategory of \emph{overconvergent $F$-isocrystals}. Only the overconvergent category admits a cohomology theory with good finiteness properties, and it is these objects that stand in for lisse Weil sheaves in the theory of companions and compatible systems
\cite{kedlaya-companions1, kedlaya-companions2}. However, the convergent category is in many ways easier to work with, and the additional structures that it provides (e.g.,
the theory of slope filtrations) are themselves of great interest.

We are thus led to the question of determining whether a given convergent $F$-isocrystal is actually overconvergent. We prove a partial result towards \cite[Conjecture~5.17]{kedlaya-isocrystals}, which asserts that a convergent $F$-isocrystal is overconvergent if and only if its restriction to every smooth curve in $X$ is overconvergent (see also Question~\ref{Q:cut by curves}).
A result in this direction was previously given by Shiho \cite{shiho-curves} under an extra hypothesis not of geometric nature: it is required that when the isocrystal is realized as a vector bundle with integrable connection on the Raynaud generic fiber of some smooth lift of $X$, the data of the vector bundle and connection can be extended to some strict neighborhood of this rigid analytic space. The content of the theorem is then to establish that the result is indeed an overconvergent isocrystal, which is to say that the formal Taylor isomorphism converges on some strict neighborhood of the tube of the diagonal. (Shiho also assumes that $k$ is uncountable rather than finite, but this restriction is not essential to the method of proof; see the discussion in \cite[\S 5]{kedlaya-isocrystals}.)

Our main result in this direction (Theorem~\ref{T:extend}) imposes an extra hypothesis of a more geometric nature: we assume that that the restrictions to curves all have tame ramification after pullback from $X$ along some single dominant morphism. This condition is inherited from the study of \emph{companions} of \'etale and crystalline coefficient objects, where it is needed in order to control the cohomology groups of coefficient objects on curves
(see the introductions of \cite{kedlaya-companions1, kedlaya-companions2} for background).
Using this result, we first obtain a weaker form of our main result (Lemma~\ref{L:apply companions}) in which the overconvergent $F$-isocrystal does not necessarily restrict to the original convergent $F$-isocrystal, but does have matching characteristic polynomials. To upgrade this result, we make a careful study of the restrictions of convergent and overconvergent $F$-isocrystals to curves; a sample result is that two convergent $F$-isocrystals are isomorphic if and only if this is true after restriction to a ``large enough'' family of curves 
(Corollary~\ref{C:restriction to enough curves}).

\section{Convergent isocrystals}

We first recall some definitions and properties of convergent $F$-isocrystals, largely in order to set notation.

\begin{notation}
Throughout the paper, let $k$ be a finite field of characteristic $p>0$ and let $X$ be a smooth, geometrically irreducible $k$-scheme (as in the introduction).
Let $K$ denote the fraction field of $W(k)$, which is a finite extension of $\QQ_p$.
Let $X^\circ$ be the set of closed points of $X$.
Let $T_X$ be the tangent bundle of $X$, also viewed as a $k$-scheme via the structure morphism to $X$.
We typically write $U$ to denote an indeterminate open dense subscheme of $X$.
By a \emph{smooth curve in $X$}, we will always mean a locally closed subscheme of $X$ which is smooth of dimension 1 and geometrically irreducible over $k$.
\end{notation}

\begin{defn} \label{D:convergent F-isocrystal}
Let $\FIsoc(X)$ denote the category of convergent $F$-isocrystals on $X$; it is a $K$-linear abelian tensor category with unit object denoted by $\calO$.
As per  \cite[Definition~2.1]{kedlaya-isocrystals}, when $X$ is affine we may fix a smooth affine formal scheme $P$ over $W(k)$ with $P_k \cong X$ and a Frobenius lift $\sigma$ on $P$,
and then identify $\FIsoc(X)$ with the category of finite projective $\Gamma(P,\calO)[p^{-1}]$-modules equipped with an integrable $K$-linear connection and a horizontal
isomorphism with the $\sigma$-pullback.

For $\calE \in \FIsoc(X)$, let $H^i(X, \calE)$ denote the de Rham cohomology of $\calE$, equipped with its Frobenius action; 
we view $H^i(X, \calE)$ as a $K$-vector space with no specified topology. Define also
\[
H^0_F(X, \calE) := \Hom_{\FIsoc(X)}(\calO, \calE), \qquad H^1_F(X, \calE) := \Ext^1_{\FIsoc(X)}(\calO, \calE). 
\]
We then have a natural identification
\[
H^0_F(X, \calE) \cong H^0(X, \calE)^{\varphi}
\]
and a natural short exact sequence (Hochschild--Serre)
\[
H^0(X, \calE)_{\varphi} \to H^1_F(X, \calE) \to H^1(X, \calE)^{\varphi}.
\]
\end{defn}

\begin{lemma} \label{L:fully faithful}
For $\calE \in \FIsoc(X)$, the natural map $H^0_F(X, \calE) \to H^0_F(U, \calE)$ is an isomorphism. Consequently (by considering internal Homs in $\FIsoc(X)$),
the restriction functor $\FIsoc(X) \to \FIsoc(U)$ is fully faithful.
\end{lemma}
\begin{proof}
See \cite[Theorem~5.3]{kedlaya-isocrystals}. 
\end{proof}
\begin{cor} \label{C:fully faithful}
For $\calE \in \FIsoc(X)$, the natural map $H^1_F(X, \calE)  \to H^1_F(U, \calE)$ is injective.
\end{cor}
\begin{proof}
This follows from Lemma~\ref{L:fully faithful}: an exact sequence in $\FIsoc(X)$ splits if and only if its restriction to $\FIsoc(U)$ does so.
\end{proof}

\begin{defn}
Using the Dieudonn\'e--Manin decomposition of convergent $F$-isocrystals over a geometric point, we may associate to $\calE \in \FIsoc(X)$ its \emph{Newton polygon function};
this function assigns to each (not necessarily closed) point of $X$ the graph of a convex, piecewise affine function on the interval $[0, \rank(\calE)]$. See \cite[\S 3]{kedlaya-isocrystals} for more information.

An important case is when the Newton polygon of $\calE$ is a constant function whose value is a polygon with a single slope $\mu$; in this case we say that $\calE$ is
\emph{isoclinic} of slope $\mu$. In the special case where $\calE$ is isoclinic of slope $0$, we also say that $\calE$ is \emph{unit-root}.
\end{defn}

\begin{lemma} \label{L:slope filtration}
For $\calE \in \FIsoc(X)$, there exists a choice of $U$ (depending on $\calE$) on which $\calE$ has constant Newton polygon. Moreover, in $\FIsoc(U)$ there is a unique filtration
\[
0 = \calE_0 \subset \cdots \subset \calE_l = \calE|_U
\]
in which each successive quotient $\calE_i/\calE_{i-1}$ is isoclinic of slope $\mu_i$
and $\mu_1 < \cdots < \mu_l$; this filtration is called the \emph{slope filtration} of $\calE$.
\end{lemma}
\begin{proof}
See \cite[Theorem~3.12, Corollary~4.2]{kedlaya-isocrystals}.
\end{proof}

\begin{defn}
In the notation of Lemma~\ref{L:slope filtration}, if $\mu_1 = 0$ then we refer to $\calE_1$ as the \emph{unit root} of $\calE$.
\end{defn}

\section{Exhaustive sequences of curves}

We next specify a key property of infinite sequences of curves to the effect that they ``fill out $X$''.

\begin{defn} \label{D:exhaustive sequence}
Let $C_1,C_2,\dots$ be a sequence of smooth curves in $X$. 
We say that this sequence is \emph{exhaustive} if it satisfies the following condition.
\begin{enumerate}
\item[(a)]
For each closed point $x \in X$, there exists $N>0$ such that $x \in C_i$ for all $i \geq N$.
Equivalently, for each finite set $S$ of closed points, there exists $N>0$ such that $S \subseteq C_i$ for all $i \geq N$.
\item[(b)]
For each finite set $S$ of closed points of $T_X$ which maps injectively to $X$,
there exist infinitely many indices $i$ such that $S \subseteq T_{C_i}$. (Note that we cannot require this for all sufficiently large $N$ because some of these conditions
are incompatible.)
\end{enumerate}
Such sequences always exist; see Corollary~\ref{C:existence of exhaustive sequence}.
\end{defn}

The condition on tangent vectors in Definition~\ref{D:exhaustive sequence} is included to make the following statement true.
\begin{lemma} \label{L:detect p-th power}
Suppose that $X = \Spec R$ is affine and let $C_1,C_2,\dots$ be an exhaustive sequence of smooth curves in $X$.
If $f \in R$ restricts to a $p$-th power on every $C_i$, then $f \in R^p$.
\end{lemma}
\begin{proof}
It is equivalent to show that $df = 0$, which may be tested at the level of tangent vectors at closed points.
\end{proof}

\begin{lemma} \label{L:refine exhaustive sequence}
Let $\calP$ be a subset of the set of curves in $X$ (i.e., a ``property'' of curves in $X$).
Let $C_1,C_2,\dots$ be a sequence of curves in $X$.
Suppose that for every finite subset $S$ of $T_X$ which maps injectively to $X$, 
every cofinite subsequence of $C_1,C_2,\dots$ contains a curve $C \in \calP$
with $S \subset T_C$.
Then there exists a sequence of indices $i_1 < i_2 < \cdots$ such that $C_{i_j} \in \calP$ for all $j$ and 
 $C_{i_1}, C_{i_2},\dots$ is an exhaustive sequence of smooth curves in $X$.
\end{lemma}
\begin{proof}
Since $X$ is of finite type over a finite field, $X^\circ$ is countable;
let $x_1, x_2,\dots$ be an enumeration of $X^\circ$.
By the same token, $(T_X)^\circ$ is countable,
as is the set $V$ of finite subsets of $(T_X)^\circ$ which map injectively to $X$; let $S_1, S_2, \dots$ be a sequence of elements of $V$
including each element of $V$ infinitely many times.

We choose the indices $i_1,i_2,\dots$ as follows. Given $i_1,\dots,i_j$, by
hypothesis we can find $i_{j+1} \notin \{i_1,\dots,i_j\}$ such that:
\begin{itemize}
\item
$\{x_1,\dots,x_{j+1}\} \subset C_{i_{j+1}}$;
\item
$S_{j+1} \subset T_{C_{i_{j+1}}}$;
\item
$C_{i_{j+1}} \in \calP$.
\end{itemize}
The resulting sequence $C_{i_1}, C_{i_2},\dots$ is then an exhaustive sequence of smooth curves in $X$.
\end{proof}
\begin{cor} \label{C:existence of exhaustive sequence}
If $X$ is quasiprojective, then there exists an exhaustive sequence of smooth curves in $X$.
\end{cor}
\begin{proof}
We apply Lemma~\ref{L:refine exhaustive sequence} with $\calP$ being the set of all curves in $X$,
and $C_1, C_2,\dots$ being an enumeration of $\calP$.
The remaining hypothesis of Lemma~\ref{L:refine exhaustive sequence} holds by Poonen's Bertini theorem \cite{poonen}.
\end{proof}

\begin{lemma} \label{L:admits a section}
Let $C_1,C_2,\dots$ be an exhaustive sequence of smooth curves in $X$.
Let $f: Y \to X$ be a finite \'etale covering. Then $f$ admits a section if and only if the induced map $Y \times_X C_i \to C_i$ admits a section for each $i$.
\end{lemma}
\begin{proof}
We may assume that $Y$ is connected and not isomorphic to $X$, in which case we wish to show that $Y \times_X C_i \to C_i$ fails to admit a section for some $i$.
Choose a geometric point $\overline{x} \to X$; then $f$ corresponds to a continuous homomorphism $\rho: \pi_1(X, \overline{x}) \to S_n$
for some $n>1$ with transitive image. For each $i$ such that $C_i$ contains the image of $\overline{x}$ in $X$,
the restriction $Y \times_X C_i \to C_i$ admits a section if and only if the action of $\rho(\pi_1(C_i, \overline{x}))$ 
on $\{1,\dots,n\}$ has a fixed point. However, by the Chebotar\"ev density theorem (e.g., see \cite[Theorem~1.1]{meagher}), 
as $i$ varies these images fill out the entire image of $\rho$,
and this image contains at least one fixed-point-free element (the average number of fixed points is $1$ and the identity exceeds this).
\end{proof}

\section{Restrictions of convergent isocrystals to curves}

Throughout this section, let $C_1,C_2,\dots$ be an exhaustive sequence of smooth curves in $X$.
We relate various properties of a convergent $F$-isocrystal on $X$ to the restrictions of this isocrystal to the $C_i$.

\begin{lemma} \label{L:H1 injective}
Suppose that $\calE \in \FIsoc(X)$ has the property that at every point of $X$, the Newton polygon of $\calE$ has all slopes negative. Then the map 
\[
H^1_F(X, \calE) \to \prod_i H^1_F(C_i, \calE|_{C_i})
\]
is injective.
\end{lemma}
\begin{proof}
By Corollary~\ref{C:fully faithful}, we may check the claim after passing from $X$ to $U$. Hence by Lemma~\ref{L:slope filtration}, we may assume that $\calE$ admits a slope filtration;
we may then reduce immediately to the case where $\calE$ is isoclinic of negative slope. We may also assume that $X$ is affine,
and thus realize the category $\FIsoc(X)$ as in Definition~\ref{D:convergent F-isocrystal}; by shrinking $X$ again, we may further ensure that $\calE$ is a free $\Gamma(P,\calO)$-module, and fix a basis.

Let $A$ be the matrix over $\Gamma(P, \calO)[p^{-1}]$ via which $\sigma$ acts on the chosen basis of $\calE$.
For each positive integer $n$, $\sigma^n$ acts on the same basis via the matrix $A_n := A \sigma(A) \cdots \sigma^{n-1}(A)$.
By a lemma of Katz \cite[Sharp Slope Estimate 1.5.1]{katz-slope}, for some $n$, $A_n^{-1}$ has entries in $p \Gamma(P, \calO)$.

An element of $H^1_F(X, \calE)$ is represented by a pair $(\bv_1, \bv_2) \in \calE \times (\calE \otimes_{\calO} \Omega^1_{P/W(k)})$ satisfying
$\nabla \bv_1 = (\sigma-1) \bv_2$; this element is zero if and only if there exists $\bw \in \calE$ such that $\bv_1 = (\sigma-1)\bw, \bv_2 = \nabla \bw$. 
Define 
\[
\bv_{1,n} := \bv_1 + \sigma \bv_1 + \cdots + \sigma^{n-1} \bv_1 \qquad (= (\sigma^n-1)\bw).
\]
Let $v_1, v_{1,n}$, and $w$ (if it exists) be the column vectors representing $\bv_1, \bv_{1,n}$, and $\bw$  respectively, in terms of our chosen basis; we may normalize so that $v_{1,n}$ has entries in $\Gamma(P, \calO)$ (without denominators).
Then the equation $\bv_1 = (\sigma-1)\bw$ translates into $v_{1,n} = (A_n \sigma^n - 1) w$ or equivalently $A_n^{-1} v_{1,n} = (\sigma^n - A_n^{-1})w$; in particular, if we can solve for $w$, then $v_{1,n}$ reduces to a $p^n$-th power modulo $p$.
By Lemma~\ref{L:detect p-th power}, we can infer this condition from
the corresponding condition on curves. Consequently, we can translate our original pair $(\bv_1, \bv_2)$ without changing the resulting class in $H^1_F(X, \calE)$ so that the vector $v_{1,n}$ becomes zero modulo $p$.

Repeating the argument, we can make $v_{1,n}$ vanish modulo successively higher powers of $p$, and in the limit reduce to a case in which $\bv_{1,n} = 0$.
At this point, we have the zero element of $H^1_F(X, \calE)$ if and only if $\bv_2 = 0$;
this vanishing can again be detected by restriction to the $C_i$.
\end{proof}

\begin{lemma} \label{L:H0 preservation}
For $\calE \in \FIsoc(X)$ which is unit-root,
there exists an exhaustive subsequence $C_{i_1}, C_{i_2}, \dots$ of $C_1,C_2,\dots$ such that
$H^0_F(X, \calE) \cong H^0_F(C_{i_j}, \calE|_{C_{i_j}})$ for each $j$.
\end{lemma}
\begin{proof}
We start with a reduction to the case where $H^0_F(X, \calE) = 0$.
Suppose that 
\[
0 \to \calE_1 \to \calE \to \calE_2 \to 0
\]
is a short exact sequence for which the conclusion holds for $\calE_1$ and $\calE_2$ (for any choice of the original exhaustive sequence).
We may then refine the sequence of curves first to get a subsequence that works for $\calE_1$, and then again to get a subsequence
that also works for $\calE_2$. At this point, Lemma~\ref{L:H1 injective} (applied to $\calE_1$)
and the five lemma imply that the conclusion also holds for $\calE$. Since the claim holds trivially if $\calE$ is spanned by $H^0_F(X, \calE)$,
by induction on $\rank \calE$ we obtain the desired reduction.

Assume now that  $H^0_F(X, \calE) = 0$.
We first produce a single index $i$ for which $H^0_F(C_i, \calE|_{C_i}) = 0$,
assuming only condition (a) of the definition of an exhaustive sequence.
By \cite[Theorem~3.7]{kedlaya-isocrystals}, we may identify $\calE$ with an \'etale $\QQ_p$-local system on $X$; by choosing a lattice,
we may obtain this from an \'etale $\ZZ_p$-local system.
For some $n$, the reduction of this local system modulo $p^n$ admits no section;
applying Lemma~\ref{L:admits a section} shows that $H^0_F(C_{i}, \calE|_{C_{i}}) = 0$ for some $i$.

Since we only used condition (a) of the definition of an exhaustive sequence, 
we can also force $C_i$ to contain a particular finite set of tangent vectors.
By Lemma~\ref{L:refine exhaustive sequence}, we can thus pick out an exhaustive subsequence.
\end{proof}

\begin{lemma} \label{L:H0 on curves}
For $\calE \in \FIsoc(X)$, there exists an exhaustive subsequence $C_{i_1}, C_{i_2}, \dots$ of $C_1,C_2,\dots$ such that
$H^0_F(X, \calE) \cong H^0_F(C_{i_j}, \calE|_{C_{i_j}})$ for each $j$.
\end{lemma}
\begin{proof}
By Lemma~\ref{L:fully faithful}, we may again check the claim after passing from $X$ to $U$.
Hence by Lemma~\ref{L:slope filtration}, we may assume that $\calE$ admits a slope filtration.
If $\mu_l > 0$, then $H^0_F(X, \calE) = H^0_F(X, \calE_{l-1})$; we may thus reduce to the case where $\mu_l \leq 0$.
Moreover, we may assume $\mu_l < 0$, as otherwise $H^0_F(X, \calE) = H^0_F(C, \calE|_C) = 0$.
By Lemma~\ref{L:H0 preservation},  after replacing the original sequence with an exhaustive subsequence,
we may ensure that 
$H^0_F(X, \calE/\calE_{l-1}) \cong H^0_F(C_i, (\calE/\calE_{l-1})|_{C_i})$ for each $i$.

We check that this sequence has the desired property.
By Lemma~\ref{L:H1 injective}, we may lift every element of $H^0_F(X, \calE/\calE_{l-1})$
to an element of $H^0_F(X, \calE)$. Hence in the commutative diagram
\[
\xymatrix{
H^0_F(X, \calE) \ar[r] \ar[d] & H^0_F(X, \calE/\calE_{l-1}) \ar[d] \\
H^0_F(C_i, \calE|_{C_i}) \ar[r] & H^0_F(C_i, (\calE/\calE_{l-1})|_{C_i})  \\
}
\]
the top horizontal arrow and the right vertical arrow are isomorphisms. Hence the bottom horizontal arrow is surjective; since it is also necessarily injective
(because $H^0_F(C_i, \calE_{l-1}|_{C_i}) = 0$ on account of slopes), it is an isomorphism, proving that the left vertical arrow is an isomorphism. 
\end{proof}

\begin{cor} \label{C:restriction to enough curves}
For $\calE_1, \calE_2 \in \FIsoc(X)$, $\calE_1 \cong \calE_2$ if and only if $\calE_1|_{C_i} \cong \calE_2|_{C_i}$ for each $i$.
\end{cor}
\begin{proof}
We need only check the ``if'' direction.
Put $\calE := \calE_1^\dual \otimes \calE_2$;
by Lemma~\ref{L:H0 on curves}, after replacing the original sequence with an exhaustive subsequence,
we have isomorphisms $H^0_F(X, \calE) \cong H^0_F(C_i, \calE|_{C_i})$ for each $i$. Pick some $i$ and choose an element
$\bv \in H^0_F(C_i, \calE|_{C_i})$ corresponding to an isomorphism $\calE_1|_{C_i} \cong \calE_2|_{C_i}$. The corresponding element
of $H^0_F(X, \calE)$ corresponds to a morphism from $\calE_1$ to $\calE_2$ which restricts to an isomorphism on $C_i$; it must thus be an isomorphism
because the kernel has locally constant rank.
\end{proof}

\section{Convergent vs. overconvergent}

We next introduce the category of overconvergent $F$-isocrystals, again largely to set notation.

\begin{defn}
Let $\FIsoc^\dagger(X)$ be the category of overconvergent $F$-isocrystals on $X$ \cite[Definition~2.4]{kedlaya-isocrystals}; it is a $K$-linear abelian tensor category with unit object denoted by $\calO^\dagger$.
For $*^\dagger \in \FIsoc^\dagger(X)$, remove the dagger to denote the image of this object in $\FIsoc(X)$;
in light of Lemma~\ref{L:fully faithful2} below, we will view $\FIsoc^\dagger(X)$ as a full subcategory of $\FIsoc(X)$.
\end{defn}

\begin{lemma} \label{L:pullback by dominant morphism}
Let $f: Y \to X$ be a dominant morphism of smooth, geometrically irreducible $k$-schemes. Then the natural map
\[
\FIsoc^\dagger(X) \to \FIsoc(X) \times_{\FIsoc(Y)} \FIsoc^\dagger(Y)
\]
is an equivalence of categories.
\end{lemma}
\begin{proof}
See \cite[Corollary~5.9]{kedlaya-isocrystals}.
\end{proof}

\begin{defn}
For $\calE^\dagger \in \FIsoc^\dagger(X)$, let $H^i(X, \calE^\dagger)$ denote the de Rham cohomology of $\calE^\dagger$, equipped with its Frobenius action; 
it is a finite-dimensional $K$-vector space \cite[Theorem~8.4]{kedlaya-isocrystals}.
Define also
\[
H^0_F(X, \calE^\dagger) = \Hom_{\FIsoc^\dagger(X)}(\calO^\dagger, \calE^\dagger), \qquad H^1_F(X, \calE^\dagger) = \Ext^1_{\FIsoc^\dagger(X)}(\calO^\dagger, \calE^\dagger). 
\]
We then have a natural identification
\[
H^0_F(X, \calE^\dagger) \cong H^0(X, \calE^\dagger)^{\varphi}
\]
and a natural short exact sequence (Hochschild--Serre)
\begin{equation} \label{eq:H1F overconvergent}
H^0(X, \calE^\dagger)_{\varphi} \to H^1_F(X, \calE^\dagger) \to H^1(X, \calE^\dagger)^{\varphi}.
\end{equation}
\end{defn}

\begin{lemma} \label{L:fully faithful2}
For $\calE^\dagger \in \FIsoc^\dagger(X)$, the natural map $H^0_F(X, \calE^\dagger) \to H^0_F(X, \calE)$ is an isomorphism. Consequently (by considering internal Homs in $\FIsoc^\dagger(X)$),
the restriction functor $\FIsoc^\dagger(X) \to \FIsoc(X)$ is fully faithful. (We will hereafter identify $\FIsoc^\dagger(X)$ with a full subcategory of $\FIsoc(X)$.)
\end{lemma}
\begin{proof}
See again \cite[Theorem~5.3]{kedlaya-isocrystals}.
\end{proof}

\begin{cor} \label{C:overconvergent to convergent injective H1}
For $\calE^\dagger \in \FIsoc^\dagger(X)$, the natural map $H^1_F(X, \calE^\dagger) \to H^1_F(X, \calE)$ is injective.
\end{cor}
\begin{proof}
This follows from Lemma~\ref{L:fully faithful2} by analogy with Corollary~\ref{C:fully faithful}.
\end{proof}

\begin{lemma} \label{L:lift filtration}
Choose $\calE^\dagger \in \FIsoc^\dagger(X)$. Suppose that $\calE$ admits a filtration in $\FIsoc(X)$
\[
0 = \calE_0 \subset \cdots \subset \calE_l = \calE
\]
such that for $i=1,\dots,l$, $\calE_i/\calE_{i-1}$ is the restriction of an object of $\FIsoc^\dagger(X)$. Then the filtration itself lifts to a filtration 
of $\calE^\dagger$ in $\FIsoc^\dagger(X)$.
\end{lemma}
\begin{proof}
By hypothesis, $\calE_1$ lifts to some $\calE_1^\dagger \in \FIsoc^\dagger(X)$.
By Lemma~\ref{L:fully faithful2}, the morphism $\calE_1 \to \calE$ lifts to a morphism $\calE_1^\dagger \to \calE^\dagger$ in $\FIsoc^\dagger(X)$, which is again an inclusion.
We may thus pass to $\calE^\dagger/\calE_1^\dagger$ and argue by induction on $l$ to conclude.
\end{proof}

\section{Logarithmic $F$-isocrystals}

As an extra tool to study overconvergent $F$-isocrystals, we introduce logarithmic convergent $F$-isocrystals.

\begin{defn}
Throughout this section, assume that $X$ admits a good compactification $\overline{X}$ with boundary $Z$ (meaning that $\overline{X}$ is smooth and $Z$ is a strict normal crossings divisor on $\overline{X}$).
We view $\overline{X}$ as a log-scheme using the logarithmic structure defined by $Z$.

Let $\FIsoc(\overline{X}, Z)$ denote the category of convergent log-$F$-isocrystals on $\overline{X}$ with nilpotent residues; it is a $K$-linear abelian tensor category with unit object denoted by $\calO$. There is a restriction functor 
$\FIsoc(\overline{X}, Z) \to \FIsoc^\dagger(X)$, via which (in light of Lemma~\ref{L:fully faithful3} below) we will identify
$\FIsoc(\overline{X}, Z)$ with a full subcategory of $\FIsoc^\dagger(X)$; in the language of \cite{kedlaya-companions1, kedlaya-companions2} these are the
\emph{docile} objects of $\FIsoc^\dagger(X)$.

For $\calE \in \FIsoc(\overline{X}, Z)$, let $H^i(\overline{X}, \calE)$ denote the de Rham cohomology of $\calE$, equipped with its Frobenius action; define also
\[
H^0_F(\overline{X}, \calE) := \Hom_{\FIsoc(\overline{X}, Z)}(\calO, \calE), \qquad H^1_F(\overline{X}, \calE) := \Ext^1_{\FIsoc(\overline{X}, Z)}(\calO, \calE). 
\]
We then have a natural identification
\[
H^0_F(\overline{X}, \calE) \cong H^0(\overline{X}, \calE)^{\varphi}
\]
and a natural short exact sequence
\[
H^0(\overline{X}, \calE)_{\varphi} \to H^1_F(\overline{X}, \calE) \to H^1(\overline{X}, \calE)^{\varphi}.
\]
\end{defn}

\begin{lemma} \label{L:fully faithful3}
For $\calE \in \FIsoc(\overline{X}, Z)$, the natural map $H^0_F(\overline{X}, \calE) \to H^0_F(X, \calE)$ is an isomorphism. Consequently (by considering internal Homs in $\FIsoc(\overline{X}, Z)$),
the restriction functor $\FIsoc(\overline{X}, Z) \to \FIsoc^\dagger(X)$ is fully faithful.
\end{lemma}
\begin{proof}
See \cite[Theorem~7.3]{kedlaya-isocrystals}.
\end{proof}
\begin{cor} \label{C:log to overconvergent injective H1}
For $\calE \in \FIsoc(\overline{X}, Z)$, the natural map $H^1_F(\overline{X}, \calE) \to H^1_F(X, \calE)$ is injective.
\end{cor}
\begin{proof}
Immediate from Lemma~\ref{L:fully faithful3}.
\end{proof}

\begin{lemma} \label{L:log extensions}
The subcategory $\FIsoc(\overline{X}, Z)$ of $\FIsoc^\dagger(X)$ is closed under formation of subquotients and extensions.
\end{lemma}
\begin{proof}
See \cite[Remark~1.4.2]{kedlaya-companions1} and references therein.
\end{proof}

\section{A local cut-by-curves calculation}

\begin{defn}
For $R$ a ring and $I$ a finitely generated ideal of $R$, we write $R^\wedge_I$ to denote the $I$-adic completion of $R$.
This is not to be confused with localization at $I$ followed by completion, which we write as 
$(R_I)^\wedge_{I}$.
\end{defn}

\begin{lemma} \label{L:local fully faithful}
Let $R$ be a $p$-adically complete ring such that $R/pR$ is reduced, equipped with a Frobenius lift $\sigma_R$.
Extend $\sigma_R$ to a Frobenius lift $\sigma$ on $R \llbracket t \rrbracket$ so that $\sigma(t) = t^p$.
Let $M$ be a finite projective $R\llbracket t \rrbracket[p^{-1}]$ equipped with an $R$-linear logarithmic connection with nilpotent residue
and a horizontal isomorphism with its $\sigma$-pullback.
Then the map
\[
H^1_F(M) \to H^1_F(M \otimes_{R \llbracket t \rrbracket} R((t))^\wedge_{(p)})
\]
is injective.
\end{lemma}
\begin{proof}
By embedding $R/p$ into a product of fields, we immediately reduce to the case where $R/pR$ is a field.
If $R/pR$ is itself perfect, we may apply Corollary~\ref{C:log to overconvergent injective H1}.
In the general case, let $\ell$ be the perfect closure of $R/pR$; we then have a $\sigma$-equivariant embedding $R \to W(\ell)$. 
We will show that 
\[
H^1_F(M) \to H^1_F(M \otimes_{R \llbracket t \rrbracket} W(\ell)\llbracket t \rrbracket)
\]
is injective, which combined with the previous argument will yield the claim.

An element of $H^1_F(M)$ can be represented by a pair $(\bv, \bw) \in M \times M$ with $(\sigma-1)(\bv) = pt\frac{d}{dt}(\bw)$. If such an element maps to zero in
$ H^1_F(M \otimes_{R \llbracket t \rrbracket} W(\ell)\llbracket t \rrbracket)$, then there exists $\bx \in M \otimes_{R \llbracket t \rrbracket} W(\ell)\llbracket t \rrbracket$
such that $t \frac{d}{dt}(\bx) = \bv$ and $(\sigma-1)(\bx) = \bw$. However, by working over $R[p^{-1}]((t))$, we see that
the condition $t \frac{d}{dt}(\bx) = \bv$ implies that $\bx$ equals an element of $H^0(M \otimes_{R \llbracket t \rrbracket} W(\ell)\llbracket t \rrbracket) = 
H^0(M) \otimes_R W(\ell)$ (determined by the constant terms of $\bv$ and $\bx$) plus an element of $M$. We may thus reduce to the case where $\bv = 0$,
in which case we may deduce the claim from the fact that $H^0_F(M) \to H^0_F(M\otimes_{R \llbracket t \rrbracket} W(\ell)\llbracket t \rrbracket)$ is injective
(again by inspection of constant terms).
\end{proof}

\begin{lemma} \label{L:local intersections}
Let $S \to S'$ be a faithfully flat morphism of $p$-complete, $p$-torsion-free rings.
Let $M$ be a finite projective $S \llbracket t \rrbracket[p^{-1}]$-module equipped with a $S$-linear connection for the derivation $t \frac{d}{dt}$
with nilpotent residue and a compatible Frobenius structure.
Then the square
\[
\xymatrix{
H^1_F(M) \ar[r] \ar[d] & H^1_F(M \otimes_{S \llbracket t \rrbracket} S((t))^{\wedge}_{(p)}) \ar[d] \\
H^1_F(M \otimes_{S \llbracket t \rrbracket} S' \llbracket t \rrbracket) \ar[r] & H^1_F(M \otimes_{S \llbracket t \rrbracket} S'((t))^{\wedge}_{(p)})
}
\]
is cartesian.
\end{lemma}
\begin{proof}
By faithfully flat descent, the \v{C}ech--Alexander sequence
\[
0 \to S \to S^1 := S' \to S^{2} := S' \widehat{\otimes} S' \to \cdots
\]
is exact. From this it follows that if we form the complexes $\Omega^\bullet_{S^*}$ and $\Omega^{\bullet}_{S^*,1}$ which compute $(F, \nabla)$-cohomology
of $M \otimes_{S \llbracket t \rrbracket} S^* \llbracket t \rrbracket$
and $M \otimes_{S \llbracket t \rrbracket} S^*((t))^\wedge_{(p)}$, respectively, then for each $i$ we have an exact sequence
\begin{equation} \label{eq:exactness to differentials}
0 \to \Omega^i_{S} \to \Omega^i_{S^1} \to \Omega^i_{S^2} \to \cdots
\end{equation}
Also, using the fact that $H^0(M \otimes_{S \llbracket t \rrbracket} S^*((t))^\wedge_{(p)})$ can be identified with the kernel of the residue map,
we deduce that
\begin{equation} \label{eq:exactness to differentials2}
0 \to H^0(\Omega^\bullet_S) \to H^0(\Omega^\bullet_{S^1}) \to H^0(\Omega^\bullet_{S^2}) \to \cdots 
\end{equation}
is exact.

Let us extend the original diagram by one row. In the new notation, it becomes
\[
\xymatrix{
H^1(\Omega^\bullet_S) \ar[r] \ar[d] & H^1(\Omega^\bullet_{S,1}) \ar[d] \\
H^1(\Omega^\bullet_{S^1}) \ar[r] \ar[d] & H^1(\Omega^\bullet_{S^1,1}) \ar[d] \\
H^1(\Omega^\bullet_{S^2}) \ar[r]  & H^1(\Omega^\bullet_{S^2,1})
}
\]
The rows of this diagram are all injective by Lemma~\ref{L:local fully faithful}, while each column composes to zero.
Consequently, if we have $\bv \in H^1(\Omega^i_{S,1})$, $\bw \in H^1(\Omega^\bullet_{S^1})$ with the same image in
$H^1(\Omega^\bullet_{S^1,1})$, then the image of $\bw$ in $H^1(\Omega^\bullet_{S^2})$ must vanish. It will thus suffice to check that
\[
H^1(\Omega^\bullet_S) \to H^1(\Omega^\bullet_{S^1}) \to H^1(\Omega^\bullet_{S^2})
\]
is exact. Using the exactness of \eqref{eq:exactness to differentials}, this reduces to the exactness of \eqref{eq:exactness to differentials2}.
\end{proof}

\begin{remark} \label{R:local intersections}
The use case we have in mind for Lemma~\ref{L:local intersections} is the following.
Let $P$ be a smooth formal scheme over $W(k)$ and put $S := \calO(P)$.
For each $x \in P_k^\circ$, let $R_x$ be the completion of $R$ along $x$.
Then $R \to R_x$ is flat for each $x$;
since any product of flat modules over a noetherian ring is flat \cite[Theorem~2.1]{chase},
$R \to \prod_x R_x$ is faithfully flat.
\end{remark}

\begin{lemma} \label{L:fully faithful log}
Let $P$ be a smooth formal scheme over $W(k)$ and put $S := \calO(P)$.
Let $M$ be a finite projective $S \llbracket t \rrbracket[p^{-1}]$-module equipped with an $S$-linear connection for the derivation $t \frac{d}{dt}$ with nilpotent residue and a compatible Frobenius structure. Then 
\[
H^0_F(M) = H^0_F(M \otimes_{S \llbracket t \rrbracket} S ((t))^{\wedge}_{(p)}).
\]
\end{lemma}
\begin{proof}
Put $\frako_L = (S_{(p)})^\wedge_{(p)}$; this is a complete discrete valuation ring with fraction field $L = \frako_L[p^{-1}]$. It is evident that
\[
S((t))^\wedge_{(p)}[p^{-1}] \cap \frako_L \llbracket t \rrbracket[p^{-1}] = S \llbracket t \rrbracket[p^{-1}]
\]
(taking the intersection in $\frako_L ((t))^{\wedge}_{(p)}[p^{-1}])$; this in turn implies that
\[
(M \otimes_{S \llbracket t \rrbracket} S((t))^\wedge_{(p)}) \cap (M \otimes_{S \llbracket t \rrbracket} \frako_L \llbracket t \rrbracket) = M
\]
(taking the intersection in $M \otimes_{S \llbracket t \rrbracket} \frako_L ((t))^{\wedge}_{(p)}$).
It thus suffices to check that
\[
H^0_F(M \otimes_{S \llbracket t \rrbracket} \frako_L \llbracket t \rrbracket) = H^0_F(M \otimes_{S \llbracket t \rrbracket} \frako_L ((t))^{\wedge}_{(p)});
\]
we may deduce this from \cite[Theorem~20.3.5]{kedlaya-book}.
\end{proof}

\begin{lemma} \label{L:add sideways connection}
Let $P$ be a smooth formal scheme over $W(k)$ and put $S := \calO(P)$.
Let $M$ be a finite projective $S \llbracket t \rrbracket[p^{-1}]$-module equipped with an $S$-linear connection for the derivation $t \frac{d}{dt}$ with nilpotent residue and a compatible Frobenius structure.
Suppose that the action of $t \frac{d}{dt}$ on $M \otimes_{S \llbracket t \rrbracket} S ((t))^{\wedge}_{(p)}$ extends to a $W(k)$-linear connection compatible with the Frobenius structure. Then this extension induces a $W(k)$-linear logarithmic (for the log structure defined by $t$) connection on $M$.
\end{lemma}
\begin{proof}
Put $M_1 := M \otimes_{S \llbracket t \rrbracket} S ((t))^{\wedge}_{(p)}$.
The $W(k)$-linear connection on $M_1$ induces a morphism $M_1 \to M_1 \otimes_S \Omega_{S/W(k)}$ which is compatible with the actions of $t \frac{d}{dt}$ and the Frobenius structures on both sides. We may thus identify it with an element of $H^0_F(M_1^\dual \otimes M_1 \otimes_S \Omega_{S/W(k)})$;
by Lemma~\ref{L:fully faithful log} the latter equals $H^0_F(M^\dual \otimes M \otimes_S \Omega_{S/W(k)})$.
Going in reverse, this element gives a morphism $M \to M\otimes_S \Omega_{S/W(k)}$ which is compatible with the actions of $t \frac{d}{dt}$ and the Frobenius structures on both sides,
which in turn yields the desired connection on $M$.
\end{proof}

\begin{lemma} \label{L:overconvergent extension from curves}
Suppose that $X$ admits a good compactification $\overline{X}$ with boundary $Z$.
Let $C_1,C_2,\dots$ be an exhaustive sequence of smooth curves in $\overline{X}$.
Choose $\calE_1^\dagger, \calE_2^\dagger \in \FIsoc^\dagger(X)$ and consider an exact sequence
\[
0 \to \calE_1 \to \calE \to \calE_2 \to 0
\]
in $\FIsoc(X)$ such that $\calE|_{C_i \times_{\overline{X}} X} \in \FIsoc(C_i, C_i \times_{\overline{X}} Z)$ for each $i$.
Then $\calE \in \FIsoc(\overline{X},Z)$; in particular, $\calE \in \FIsoc^\dagger(X)$.
\end{lemma}
\begin{proof}
Using internal Homs, we may reduce to the case $\calE^\dagger_2 = \calO^\dagger$. That is, we are given a class $\bv \in H^1_F(X, \calE_{1})$ whose image in $H^1_F(C_i, \calE_{1})$ belongs to $H^1_F(C_i, \calE_1^\dagger)$ for every $i$, and 
we wish to conclude that $\bv \in H^1_F(X, \calE_1^\dagger)$. 
(Here we are using Corollary~\ref{C:overconvergent to convergent injective H1}
to identify $H^1_F(X, \calE_1^\dagger)$ with a $\QQ_p$-subspace of $H^1_F(X, \calE_1)$.)

Our first goal is to show that $\calE$ extends to $\FIsoc(\overline{X} \setminus W, Z \setminus W)$ where $W$ is the nonsmooth locus of $Z$. For this we may work locally on $\overline{X} \setminus W$,
so we may restrict to an open subset $U$ of $\overline{X} \setminus W$ meeting only one component of $Z$.
We may also assume that there exists $t \in \calO(U)$ cuts out $Z \cap U$ (schematically) and that there exists an isomorphism $\calO(U)^\wedge_{(t)} \cong \calO(U \cap Z) \llbracket t \rrbracket$. 

We next show that the action of $t \frac{d}{dt}$ can be extended across $U \cap Z$, which by Lemma~\ref{L:add sideways connection}
will imply that the rest of the connection also extends.
Set $S := \calO(U \cap Z)$. For each $x \in (U \cap Z)^\circ$ corresponding to a maximal ideal $\frakm$ of $S$, choose an index $i$ such that the curve $C_i$ passes through $x$ and meets $Z$ transversely (by the definition of an exhaustive sequence there are in fact infinitely many such indices).
By hypothesis, the restriction of $\calE$ to $\FIsoc^\dagger(C_i \times_{\overline{X}} X)$ extends to an object of $\FIsoc(C_i, C_i \times_{\overline{X}} Z)$. The latter in turn extends uniquely over the formal completion of $\overline{X}$ along $C_i$;
in particular, we obtain an extension of the desired form after we replace $S$ with $S^{\wedge}_{\frakm}$,
or for that matter with the product $S' := \prod_\frakm S^{\wedge}_{\frakm}$ running over maximal ideals of $S$.
As per Remark~\ref{R:local intersections}, $S'$ is faithfully flat over $S$; we may thus apply Lemma~\ref{L:local intersections} to achieve the desired effect.

By \cite[Theorem~5.1]{kedlaya-isocrystals}, any object of $\FIsoc(\overline{X} \setminus W, Z \setminus W)$ belongs to $\FIsoc^\dagger(X)$.
We can then apply \cite[Theorem~7.4]{kedlaya-isocrystals} to promote $\calE$ to $\FIsoc(\overline{X}, Z)$, as desired.
\end{proof}

\section{Separating algebraic eigenvalues}

We now make a discussion parallel to \cite[\S 8.4]{kedlaya-companions2}.

\begin{defn}
Define the formal base extensions $\FIsoc(X) \otimes \overline{\QQ}_p, \FIsoc^\dagger(X) \otimes \overline{\QQ}_p$ as in \cite[\S 9]{kedlaya-isocrystals}.
For $\calE \in \FIsoc(X) \otimes \overline{\QQ}_p$ and
$x \in X^\circ$, let $c(\calE, x)$ be the number of eigenvalues of Frobenius on $\calE_x$ which belong to $\overline{\QQ}$. 
We say that $\calE$ is \emph{algebraic} if $c(\calE,x) = \rank(\calE_x)$ for all $x \in X^\circ$.
\end{defn}

\begin{lemma} \label{L:split algebraic eigenvalues}
Let $C_1, C_2, \dots$ be an exhaustive sequence of smooth curves on $X$.
Suppose that for some $\calE \in \FIsoc(X) \otimes \overline{\QQ}_p$, for every $j$, $\calE|_{C_j} \in \FIsoc(C_j) \otimes \overline{\QQ}_p$ is the restriction of an object $\calE^\dagger|_{C_j} \in \FIsoc^\dagger(C_j) \otimes \overline{\QQ}_p$. 
\begin{enumerate}
\item[(a)]
There exists a unique direct sum decomposition
$\calE = \calE_1 \oplus \calE_2$ such that $c(\calE,x) = c_1(\calE_2,x) = \rank(\calE_{2,x})$ for all $x \in X^\circ$.
\item[(b)]
For $i=1,2$, for every $j$, $\calE_i|_{C_j} \in \FIsoc(C_j) \otimes \overline{\QQ}_p$ is the restriction of an object $\calE_i^\dagger|_{C_j} \in \FIsoc^\dagger(C_j) \otimes \overline{\QQ}_p$. 
\end{enumerate}
\end{lemma}
\begin{proof}
We may assume that $X$ is affine.
If $X$ is a curve, this is a direct consequence of \cite[Corollary~8.4.4]{kedlaya-companions2}.
For general $X$, for each $j$ we obtain a decomposition as in (a); the projector onto the first factor corresponds to an element of $H^0_F(C_j, (\calE^\dual \otimes \calE)|_{C_j})$.
By Lemma~\ref{L:H0 on curves}, on some exhaustive subsequence these projectors are the restrictions of a single endomorphism of $\calE$; this endomorphism is necessarily a projector.
The corresponding decomposition of $\calE$ satisfies (a) by construction and (b) by Lemma~\ref{L:fully faithful2}.
\end{proof}

\begin{remark} \label{R:uniformly algebraic}
For $\calE \in \FIsoc^\dagger(X) \otimes \overline{\QQ}_p$ algebraic, Deligne's construction
\cite[Theorem~3.4.2]{kedlaya-companions1} shows that $\calE$ is in fact \emph{uniformly algebraic} in the sense that the characteristic polynomials of Frobenius at all $x \in X^\circ$ all have coefficients in a single number field. This does not however imply that the \emph{eigenvalues} of Frobenius all belong to a single number field.

On a related note, the corresponding implication for $\calE \in \FIsoc(X) \otimes \overline{\QQ}_p$ fails: if $\calE$ is algebraic, it need not be uniformly algebraic. For example, let $X$ be the ordinary locus of some modular curve, let $\calF \in \FIsoc^\dagger(X) \otimes \overline{\QQ}_p$ be the tautological object of rank 2 (i.e., the middle relative cohomology of the universal elliptic curve over $X$), and let $\calE \in \FIsoc(X) \otimes \overline{\QQ}_p$ be the unit-root subobject of $\calF$ (of rank 1). Then for each $x \in \calE$ the Frobenius trace at $x$ is a quadratic irrational, but it can be shown that these traces do not generate a finite extension of $\QQ$.
As far as we know, this example illustrates the worst that can happen; see Question~\ref{Q:uniformly algebraic}.
\end{remark}

\section{Weight filtration}

\begin{defn}
Throughout this section, let $\overline{\QQ}$ denote the integral closure of $\QQ$ in $\overline{\QQ}_p$ and fix an embedding $\iota: \overline{\QQ} \to \CC$.
We say $\calE \in \FIsoc(X) \otimes \overline{\QQ}_p$ is \emph{$\iota$-pure of weight $w$} if for each $x \in X^\circ$ with residue field of order $q$,
the eigenvalues of Frobenius on $\calE_x$ all have $\iota$-absolute value $q^{w/2}$.
\end{defn}

\begin{lemma} \label{L:weight filtration}
Suppose that $\calE^\dagger \in \FIsoc^\dagger(X) \otimes \overline{\QQ}_p$ is algebraic.
\begin{enumerate}
\item[(a)]
If $\calE^\dagger$ is irreducible, then it is $\iota$-pure of some weight.
\item[(b)]
There exists a unique filtration 
\[
0 = \calE^\dagger_0 \subset \cdots \subset \calE^\dagger_l = \calE^\dagger
\]
in $\FIsoc^\dagger(X) \otimes \overline{\QQ}_p$ for which there exists an increasing sequence $w_1 < \cdots < w_n$ of real numbers for which $\calE^\dagger_i/\calE^\dagger_{i-1}$ is $\iota$-pure of weight $w_i$.
We call this filtration the \emph{weight filtration} (or more precisely the \emph{$\iota$-weight filtration}) of $\calE^\dagger$.
\end{enumerate}
\end{lemma}
\begin{proof}
See \cite[Theorem~3.1.9]{kedlaya-companions1}.
\end{proof}

\begin{lemma} \label{L:splitting pure exact sequence}
Let 
\[
0 \to \calE^\dagger_1 \to \calE^\dagger \to \calE^\dagger_2 \to 0
\]
be an exact sequence in $\FIsoc^\dagger(X) \otimes \overline{\QQ}_p$ in which $\calE^\dagger_1$ and $\calE^\dagger_2$ are both $\iota$-pure of the same weight $w$.
If $H^0_F(X, \calE_2^{\dagger \dual} \otimes \calE_1^\dagger) = 0$, then the exact sequence splits.
\end{lemma}
\begin{proof}
Since $H^0(X, \calE_2^{\dagger \dual} \otimes \calE_1^\dagger)$ is a finite-dimensional $\overline{\QQ}_p$-vector space, we have
\[
\dim_{\overline{\QQ}_p} H^0(X, \calE_2^{\dagger \dual} \otimes \calE_1^\dagger)_\varphi
=
\dim_{\overline{\QQ}_p} H^0(X, \calE_2^{\dagger \dual} \otimes \calE_1^\dagger)^\varphi
=
\dim_{\overline{\QQ}_p} H^0_F(X, \calE_2^{\dagger \dual} \otimes \calE_1^\dagger) = 0.
\]
Meanwhile, by the theory of weights  \cite[Lemma~3.1.3]{kedlaya-companions1}, we have $H^1(X, \calE_2^{\dagger \dual} \otimes \calE_1)^\varphi = 0$.
By \eqref{eq:H1F overconvergent}, we obtain the desired splitting.
\end{proof}
\begin{cor} \label{C:pure isotypical}
Suppose that $\calE^\dagger \in \FIsoc^\dagger(X)$ is $\iota$-pure of some weight. Then $\calE^\dagger$ admits an isotypical decomposition, i.e., a direct sum decomposition
in which each summand is a successive extension of copies of a single irreducible object.
\end{cor}
\begin{proof}
This follows from Lemma~\ref{L:splitting pure exact sequence} as in \cite[Corollary~3.1.4]{kedlaya-companions1}.
\end{proof}
\begin{cor} \label{C:lowest seight subobject}
Suppose that $\calE^\dagger \in \FIsoc^\dagger(X)$ is $\iota$-pure of some weight. Then there is a subobject of $\calE^\dagger$ which splits as a direct sum of pairwise
nonisomorphic irreducible subobjects representing all of the isomorphism classes of constituents of $\calE$; there is also a quotient of the same form.
\end{cor}
\begin{proof}
Immediate from Corollary~\ref{C:pure isotypical}.
\end{proof}

\section{Overconvergence by restrictions to curves}

We now address the question discussed in the introduction and prove our main result (Theorem~\ref{T:extend}). Note that in contrast with our previous work, here we cannot restrict attention to an exhaustive sequence of curves; see Question~\ref{Q:only use exhaustive sequence} and subsequent discussion.

\begin{question}  \label{Q:cut by curves}
Suppose that for some $\calE \in \FIsoc(X)$, for every smooth curve $C$ in $X$, we have $\calE|_{C} \in \FIsoc^\dagger(C)$. Does it then follow that $\calE \in \FIsoc^\dagger(X)$?
\end{question}

We start with a result that yields a weaker conclusion.

\begin{lemma} \label{L:apply companions}
Suppose that for some $\calE \in \FIsoc(X)$, for every smooth curve $C$ in $X$, $\calE|_{C}$ is the restriction of an object $\calE^\dagger|_C \in \FIsoc^\dagger(C)$.
Suppose further that there exists a dominant morphism $f: Y \to X$ such that for every smooth curve $C$ in $Y$, $f^* \calE|_C$ is a tame object of $\FIsoc^\dagger(C)$.
Then there exists $\calF^\dagger \in \FIsoc^\dagger(X) \otimes \overline{\QQ}_p$ such that $\calE_x$ and $\calF_x$ have the same characteristic polynomial of Frobenius for every $x \in X^\circ$.
(That is to say, they have the same semisimplification in $\FIsoc(x) \otimes \overline{\QQ}_p$.)
\end{lemma}
\begin{proof}
Using Lemma~\ref{L:split algebraic eigenvalues} and constant twists, we may reduce to the case where $\calE$ is algebraic, or equivalently uniformly algebraic \cite[Corollary~3.4.3]{kedlaya-companions1}.
 Then the restrictions $\calE|_C$ form a \emph{skeleton sheaf} in the sense of Drinfeld's theorem \cite[Theorem~2.5]{drinfeld-deligne};
by combining that result with the existence of crystalline companions \cite[Theorem~0.1.2]{kedlaya-companions2}, we may deduce the claim.
\end{proof}

\begin{remark}
It should be possible to extend Lemma~\ref{L:apply companions} by imposing the condition only on curves $C$ in an exhaustive sequence.
The missing ingredient to prove this is to upgrade \cite[Theorem~2.5]{drinfeld-deligne} to rely on a similar hypothesis.
\end{remark}

\begin{cor} \label{C:same semisimplification}
For $\calE, \calF^\dagger$ as in Lemma~\ref{L:apply companions}, the objects $\calE^\dagger|_C$ and $\calF^\dagger|_C$ in $\FIsoc^\dagger(C) \otimes \overline{\QQ}_p$
have the same semisimplification.
\end{cor}
\begin{proof}
Apply \cite[Theorem~3.3.1]{kedlaya-companions1}.
\end{proof}

\begin{cor} \label{C:first weight subobject}
Fix an embedding $\iota: \overline{\QQ} \to \CC$.
For $\calE, \calF^\dagger$ as in Lemma~\ref{L:apply companions} with $\calE$ algebraic, 
there is a subobject $\calF_0^\dagger$ of the first step of the $\iota$-weight filtration of $\calF^\dagger$
such that $\calF_0$ occurs as a subobject of $\calE$.
\end{cor}
\begin{proof}
By Lemma~\ref{L:weight filtration}, $\calF^\dagger$ admits an $\iota$-weight filtration, which then restricts to the $\iota$-weight filtration on $\calF^\dagger|_C$ for every smooth curve $C$ in
$X$. Let $\calF_1^\dagger$ be the first step of this filtration of $\calF^\dagger$;
for every smooth curve $C$ in $X$, Corollary~\ref{C:same semisimplification} implies that $H^0_F(C, \calF_1^{\dagger, \dual}|_C \otimes \calE^\dagger|_C) \neq 0$ and hence
$H^0_F(C, (\calF_1^\dual \otimes \calE)|_C) \neq 0$.
By Lemma~\ref{L:H0 preservation}, $H^0_F(X, \calF_1^\dual \otimes \calE) \neq 0$; we take $\calF_0$ to be the image of a nonzero morphism $\calF_1 \to \calE$.
\end{proof}
\begin{cor} \label{C:successive extension of overconvergents}
Fix an embedding $\iota: \overline{\QQ} \to \CC$.
For $\calE$ as in Lemma~\ref{L:apply companions} and algebraic, 
there exists a filtration in $\FIsoc(X) \otimes \overline{\QQ}_p$
\[
0 = \calE_0 \subset \cdots \subset \calE_l = \calE
\]
such that for $i=1,\dots,l$, the following conditions hold.
\begin{enumerate}
\item[(a)]
The object $\calE_i/\calE_{i-1}$ of $\FIsoc(X) \otimes \overline{\QQ}_p$
is the restriction of an object $\calF_i^\dagger \in \FIsoc^\dagger(X) \otimes \overline{\QQ}_p$.
\item[(b)]
There exists a sequence $w_1 \leq \cdots \leq w_l$ of real numbers such that
$\calF_i^\dagger$ is $\iota$-pure of weight $w_i$.
\item[(c)]
The restriction of the filtration to $\calE|_C$ lifts to a filtration of $\calE^\dagger|_C$ in $\FIsoc^\dagger(X) \otimes \overline{\QQ}_p$.
\end{enumerate}
\end{cor}
\begin{proof}
Parts (a) and (b) are immediate from Corollary~\ref{C:first weight subobject}.
Part (c) then follows from Lemma~\ref{L:lift filtration}.
\end{proof}

\begin{theorem} \label{T:extend}
Suppose that for some $\calE \in \FIsoc(X)$, for every smooth curve $C$ in $X$, $\calE|_{C} \in \FIsoc(C)$ is the restriction of an object $\calE^\dagger|_C \in \FIsoc^\dagger(C)$.
Suppose further that there exists a dominant morphism $f: Y \to X$ such that for every smooth curve $C$ in $Y$, $f^* \calE^\dagger|_C$ is a tame object of $\FIsoc^\dagger(C)$.
Then $\calE \in \FIsoc^\dagger(X)$.
\end{theorem}
\begin{proof}
By Lemma~\ref{L:pullback by dominant morphism}, we are free to check the claim after pullback along a dominant morphism. We may thus assume that $Y = X$.

By Lemma~\ref{L:fully faithful2}, it suffices to check that $\calE \in \FIsoc^\dagger(X) \otimes \overline{\QQ}_p$. Namely, as an object of $\FIsoc(X) \otimes \overline{\QQ}_p$,
$\calE$ is isomorphic to its conjugates by the action of $\Gal(\overline{\QQ}_p/\QQ_p)$ on coefficients, and Lemma~\ref{L:fully faithful2} implies that these isomorphisms
extend to $\FIsoc^\dagger(X) \otimes \overline{\QQ}_p$.

By Corollary~\ref{C:successive extension of overconvergents}, $\calE$ is a successive extension of 
restrictions of objects of $\FIsoc^\dagger(X) \otimes \overline{\QQ}_p$.  We may then apply Lemma~\ref{L:overconvergent extension from curves} to conclude that
$\calE \in \FIsoc^\dagger(X) \otimes \overline{\QQ}_p$.
\end{proof}

\begin{remark}
While the condition on the dominant morphism $f$ in Theorem~\ref{T:extend} corresponds to an absolutely essential condition in \cite[Theorem~2.5]{drinfeld-deligne}, we see no reason to believe that 
Question~\ref{Q:cut by curves} fails to have an affirmative answer.
Indeed, it may be possible to show that the hypothesis of Question~\ref{Q:cut by curves}  directly implies
the stronger hypothesis of Theorem~\ref{T:extend}; see the discussion below.
\end{remark}

\section{Additional questions}

We record some related questions left unanswered by our work.
The following questions are formulated with an eye towards eliminating the condition on the dominant morphism in Theorem~\ref{T:extend}.
\begin{question}
Suppose that $X$ is a curve with smooth compactification $\overline{X}$.
Suppose that $\calE \in \FIsoc(X)$ is the restriction of an object $\calE^\dagger$ of $\FIsoc^\dagger(X)$. Can one bound the Swan conductor of $\calE^\dagger$ at a point of $\overline{X} \setminus X$ in terms of some data associated to $\calE$?
\end{question}

The following is a $p$-adic analogue of a question of Deligne; see below.
\begin{question} \label{Q:uniform bound on Swan conductor}
Suppose that $X$ admits a good compactification $\overline{X}$ with boundary $Z$.
Suppose that $\calE \in \FIsoc(X)$ has the property that for each smooth curve $C$ in $\overline{X}$, $\calE|_C \in \FIsoc(C \times_{\overline{X}} X)$ is the restriction of an object of 
$\FIsoc^\dagger(C \times_{\overline{X}} X)$ the sum of whose Swan conductors is bounded by some constant times the intersection number $C \cdot Z$. Does it follow that $\calE \in \FIsoc^\dagger(X)$?
\end{question}

\begin{remark}
In \cite[Question~1.2]{esnault-kerz}, it is asked whether a skeleton sheaf on $X$ comes from an $\ell$-adic local system if it obeys a similar uniformity condition on Swan conductors.
An affirmative answer to Question~\ref{Q:uniform bound on Swan conductor} would not directly imply this, because we cannot directly assert that a skeleton sheaf gives rise to a convergent $F$-isocrystal without any hypotheses on wild ramification.
\end{remark}

The following question is motivated by a recent theorem of Tsuzuki which asserts that an overconvergent $F$-isocrystal is determined by its unit root 
\cite{tsuzuki-minimal}.

\begin{question} \label{Q:cut by curves unit root}
Let $\calE \in \FIsoc(X)$ be a unit-root object. Suppose that for every smooth curve $C$ in $X$, $\calE|_C \in \FIsoc(C)$ is the unit root of some object of $\FIsoc^\dagger(C)$. Does it follow that $\calE$ is the unit root of some object of $\FIsoc^\dagger(X)$?
\end{question}

\begin{remark}
It would be natural to try to prove Question~\ref{Q:cut by curves} under an extra tameness hypothesis as in Theorem~\ref{T:extend}.
It might also be necessary to add additional hypotheses, e.g., that the Frobenius traces of the $\calE|_C$ satisfy a uniformity condition as in Question~\ref{Q:uniformly algebraic},
or  that there is a uniform bound on the rank of the object of $\FIsoc^\dagger(C)$ admitting $\calE|_C$ as its unit root.

One tool that might be relevant for this question is Wan's proof of the meromorphicity of unit-root $L$-functions \cite{wan1, wan2, wan3}.
\end{remark}

The following question arises naturally from Remark~\ref{R:uniformly algebraic}. Part (b)
is a variant of Question~\ref{Q:cut by curves unit root}.
\begin{question} \label{Q:uniformly algebraic}
Let $\calE \in \FIsoc(X) \otimes \overline{\QQ}_p$ be algebraic. For $x \in X^\circ$, let $F_x$ be the number field generated by the coefficients of the characteristic polynomial of Frobenius at $x$. 
\begin{enumerate}
\item[(a)]
Does it follow that $\sup_{x \in X^\circ} [F_x:\QQ] < \infty$?
\item[(b)]
Does it further follow that $\calE$ occurs as a subquotient of an object of $\FIsoc^\dagger(X) \otimes \overline{\QQ}_p$?
\end{enumerate}
\end{question}

\begin{question} \label{Q:only use exhaustive sequence}
Does Theorem~\ref{T:extend} remain true if the hypotheses on $C$ only hold for the terms of an exhaustive sequence?
\end{question}

\begin{remark} \label{R:only use exhaustive sequence}
One serious difficulty in Question~\ref{Q:only use exhaustive sequence} is the fact that Deligne's argument for uniform algebraicity (see Remark~\ref{R:uniformly algebraic}) does not as written apply when one only starts with a skeleton sheaf defined on an exhaustive sequence; it may thus be necessary to add the hypothesis that $\calE$ is uniformly algebraic. It is unclear to us whether limiting the family of curves also creates an essential difficulty in Drinfeld's construction;
it may be necessary to strengthen the definition of an exhaustive sequence to incorporate a form of Hilbert irreducibility, such as \cite[Theorem~2.15]{drinfeld-deligne}.
\end{remark}

\end{document}